\documentclass[a4paper,12pt]{article}
\usepackage{graphicx} 

\usepackage{tikz}

\usepackage{amsmath,amsthm,amscd,amssymb,eucal,mathrsfs}
\usepackage{amsfonts}
\usepackage{latexsym}
\usepackage{fullpage}
\usepackage{dsfont}
\usepackage[all]{xy}

\newtheorem{thm}{Theorem}[section]

\newtheorem{remark}[thm]{Remark}
\newtheorem{definition}[thm]{Definition}

\newtheorem{example}[thm]{Example}

\newtheorem{Remark}[thm]{Remark}

\newtheorem{Proposition}[thm]{Proposition}

\newtheorem{Lemma}[thm]{Lemma}

\newtheorem{Cor}[thm]{Corollary}

\newtheorem{assumption}{Assumption}

\newcommand{\mathset}[1]{{\left\{#1\right\}}}
\newcommand{\absolute}[1]{\left\lvert#1\right\rvert}
\newcommand{\norm}[1]{\left\|#1\right\|}

\DeclareMathOperator{\children}{ch}
\DeclareMathOperator{\diam}{diam}
\DeclareMathOperator{\diag}{diag}

\DeclareMathOperator{\Trace}{Tr}
\DeclareMathOperator{\tr}{tr}
\DeclareMathOperator{\Lip}{Lip}

\title{Vladimirov-Pearson Operators on $\zeta$-regular Ultrametric Cantor Sets}
\author{Patrick Erik Bradley
\\
Institute of Photogrammetry and Remote Sensing (IPF)
\\
Geodetic Institute Karlsruhe (GIK)
\\
Karlsruhe Institute of Technology (KIT)
\\
Email: bradley@kit.edu}
\date{\today}

\begin{document}

\maketitle


\begin{abstract}
A new operator for certain types of ultrametric Cantor sets is constructed using the measure coming from the spectral triple associated with the Cantor set, as well as its zeta function. Under certain mild conditions on that measure, it is shown that it is an integral operator similar to the Vladimirov-Taibleson operator on the $p$-adic integers. Its spectral properties are studied, and the Markov property and kernel representation of the heat kernel generated by this so-called \emph{Vladimirov-Pearson} operator is shown, viewed as acting on a certain Sobolev space. A large class of these operators have a heat kernel and a Green function explicitly given by the ultrametric wavelets on the Cantor set, which are eigenfunctions of the operator. 
\end{abstract}

\section{Introduction}

For the last decades, quite some progress has been achieved in the domain of $p$-adic and ultrametric analysis, whereby the former has received a lot more attention than the latter. One reason is the regularity of the hierarchical structure of $p$-adic domains due to the additive group structure of the $p$-adic number fields together with the fact that their residue fields enforce a regularity on the trees associated with $p$-adic discs. The fractal structure of $p$-adic numbers allows the viewpoint of $p$-adic discs as examples of ultrametric Cantor sets. 
Pearson and Bellissard endow every regular ultrametric Cantor set $C$, i.e.\ every Cantor set together with an ultrametric which induces its topology, with a measure based on  by A.\ Connes coming from a spectral triple on the Cantor set \cite{PB2009}. This gives them the possibility of defining an operator on the Hilbert space of $L^2$-functions on $C$. This approach opens the door towards extending results in $p$-adic analysis to regular ultrametric Cantor sets. In particular, the study of new operators on these which give rise to Markov processes are in the interest of ongoing research by the author.
\newline

The first type of a $p$-adic operator for diffusion on the $p$-adic numbers is the Vladimirov operator, well studied in the book \cite{VVZ1994}, and which is actually already present in the book \cite{Taibleson1975}. Since then, the study of stochastic processes on the $p$-adic numbers emerged, cf.\ e.g.\
\cite{Zuniga2008,Zuniga2017} 
or the book
\cite{Zuniga2025}.
Viewing $p$-adic Brownian motion as a scaling limit has been achieved recently: on the $p$-adic numbers \cite{Weisbart2024}, on  $p$-adic vector spaces \cite{PRSWY2024}, and on the $p$-adic integers \cite{PW2024}.
These results rely on previous work on the study of diffusion processes on $p$-adic sets other than the $p$-adic number fields, of which the $p$-adic integers are a prominent example, cf.\ \cite{Kochubei2018}. Z\'u\~niga-Galindo provided for a method with which diffusion processes within a finite collection of $p$-adic discs can be studied, which jump between distinct discs \cite{ZunigaNetworks}. This became the basis of the author's recent work for defining and studying diffusion on spaces which are locally $p$-adic discs, i.e.\ the $p$-adic points of Mumford curves, cf.\ \cite{brad_HeatMumf,brad_thetaDiffusionTateCurve,SchottkyDiff}. In \cite{BL_shapes_p}, $p$-adic diffusion was used in order to hear the shape of a graph, in particular that of the skeleton of a Mumford curve, and in  
\cite{HearingGenusMumf}
it was shown that the spectrum of diffusion operators built with automorphic forms and a regular differential $1$-form are sufficient to  recover the genus of a Mumford curve.
\newline

As an application of a Vladimirov-Taibleson or Z\'u\~niga approach, 
new forms of
non-autonomous diffusion on time-dependent graphs or time-varying energy landscapes were developped \cite{nonAutonomousDiffusion,Ledezma_Energy}, and
 diffusion on finite
multi-Topology systems was approximated using ultrametrics \cite{BL-TopoIndex_p}. 
In order to extend the author's recent work to more general ultrametric spaces and include interesting algebraic examples, 
Schottky invariant diffusion on the
transcendent $p$-adic upper half plane was obtained in
\cite{SchottkyTrans_p}.
Work in progress is also the extension of 
Boundary Value Problems for $p$-Adic Elliptic Parisi-Zúñiga Diffusion \cite{ellipticBVP} to the yet to be defined ultrametric manifolds in order to transfer methods from partial differential equations to this setting. That work relies on ideas found in  \cite{BGPW2014,RW2023,RodriguezDiss,
UnitaryDual_p}.
\newline

The $\zeta$-function coming from the spectral triple associated with a regular ultrametric Cantor set $(C,d)$ is used in this article.
It is assumed that $\zeta(s)$ has an abscissa of convergence in $\mathds{R}$.
Furthermore, the Connes measure $\mu_\tau$ is assumed throughout to have the property that the discs defined by the children of a given vertex in the Michon Tree of $(C,d)$ are all of equal measure w.r.t.\ $\mu_\tau$.  
Under these assumptions, the following results are proven in this article:

\begin{enumerate}
\item The new Vladimirov-Pearson operator $\mathcal{D}^s$ with $s\in\mathds{R}$ can be represented as an integral operator on the space of test functions on a $\zeta$-regular ultrametric Cantor set $(C,d)$  [Theorem \ref{LocalCoordinateOperator}]

\item  The ultrametric wavelets defined in \cite{XK2005} are eigenfunctions of $\mathcal{D}^s$, and their corresponding eigenvalues have an explicit description similar to the ones in \cite{XK2005}. [Corollary \ref{LocalCoordinateEigenvalue}]
\item The operator $\mathcal{D}^s$ acting on $L^2(C,\mu_\tau)$ is self-adjoint and positive semi-definite. The Hilbert space $L^2(C,\mu_\tau)$ has
an orthonormal basis consisting of eigenfunctions of $\mathcal{D}^s$. The operator $\mathcal{D}^s$ is bounded if
and only if $s \ge 4$. In this case it is a compact operator, and $0$ is an accumulation point
of the spectrum of $\mathcal{D}^s$. In the case that $s < 4$, the spectrum of $\mathcal{D}^s$ is a point spectrum. [Theorem \ref{LocalCoordinateSpectrum}]
\item The operator $\mathcal{D}^s$ generates a strongly continuous Markov semigroup $e^{-t\mathcal{D}^s}$ ($t\ge0$) on a suitable Sobolev space $W^{1,2}(C)$. [Theorem \ref{MarkovSemigroup}]

\item The semigroup $e^{-t\mathcal{D}^s}$ $(t\ge0)$ has a kernel representation $p_t(x,\cdot)$ for $x\in C$. [Corollary \ref{kernelRepresentation}]

\item In the case that $s<4$, the Vladimirov-Pearson operator $-\mathcal{D}^s$ has a heat kernel which is explicitly representable with the ultrametric wavelet eigenbasis of $L^2(C,\mu_\tau$). [Theorem \ref{LocalCoordinateHeatKernel}]

\item In the case that $s<4$, a Green function exists for $-\mathcal{D}^s$ of an explicit form
using the ultrametric wavelets, provided the sum
\[
\sum\limits_{\lambda>0}\lambda^{-1}
\]
converges, with $\lambda$ ranging over all the positive eigenvalues of $\mathcal{D}^s$. [Corollary \ref{GreenFunction}]
\end{enumerate}

The following Section 2 gives a brief account on regular ultrametric Cantor sets. Section 3 is devoted to the construction and spectral properties of the Vladimirov-Pearson operators from the data provided in \cite{PB2009}. There it is also shown that the $p$-adic transcendent numbers are an example for Cantor sets having an irregular Michon tree for which a Connes measure can be given such that the assumptions of this article are are satisfied. Section 4 contains the proofs of the asserted properties of the semigroup associated with the Vladimirov-Pearson operator acting on the Sobolev space $W^{1,2}(C)$.

\section{Regular Ultrametric Cantor sets}

\begin{definition}
A \emph{Cantor set} is a totally disconnected compact metrisable space without isolated points.
\end{definition}

\begin{Remark}
Notice that up to homeomorphism, there is precisely one Cantor set.
\end{Remark}

\begin{definition}
A \emph{locally compact local ultrametric Cantor set} is a second countable Hausdorff space in which each point has an open neighbourhood which is a Cantor set.
\end{definition}

\begin{definition}
An ultrametric $d$ on a Cantor set $C$ is \emph{regular}, if $d$ generates the topology of $C$. The pair $(C,d)$ is then called a \emph{regular ultrametric Cantor set}.
\end{definition}

To a regular ultrametric Cantor set $(C,d)$, a weighted tree, called the \emph{Michon tree}, can be associated whose boundary is isometrically isomorphic with $C$, cf.\ \cite{Michon1985}.

\begin{example}
A local field $K$ is a locally compact local ultrametric Cantor set. The Michon tree associated with any disc
in $K$ is a regular rooted tree in which the number of child nodes equals the cardinality of the residue field of $K$. This is well-known.
\end{example}

\begin{example}
  The transcendent $p$-adic upper half plane $\Omega_{\tr}$
  consists of the transcendent elements of the $p$-adic upper half plane
  studied in depth in \cite{DT2007} in the context of Shimura curves, cf.\ also 
 \cite{Kuennemann2000}. It 
 is a locally compact local ultrametric Cantor set, because
 it is covered by the
 orbits of  transcendent $p$-adic numbers under the action of the $p$-adic absolute Galois group, and each of these is a pro-finite set \cite[Theorem 3.5]{APZ1998}, meaning that it can be endowed with a regular ultrametric. 
 In \cite{SchottkyTrans_p}, Schottky invariant diffusion operators are constructed on its transcendent part $\Omega_{\tr}$.
This is an example of a regular ultrametric Cantor set whose Michon tree is not regular, and even has no bound on the number of children of vertices, cf.\ the ongoing work \cite{HalwasGaloisOrbitTree}.
\end{example}

The space of test functions is denoted as $\mathcal{D}(C)$ and consists of the locally constant functions $C\to\mathds{C}$.

\section{Vladimir-Pearson operators}

In \cite{PB2009}, a spectral triple $(\mathscr{A},\mathcal{H},D)$ is associated with a regular ultrametric Cantor set $(C,d)$. Namely,
\[
\mathscr{A}=\mathcal{C}_{\Lip}(C),\quad\mathcal{H}=\ell^2(V)\otimes\mathds{C}^2,\quad D\psi(v)=\diam([v])^{-1}\sigma_1\psi(v),
\]
where $\mathcal{C}_{\Lip}(C)$ is the space of Lipshitz-continuous functions $C\to\mathds{C}$, $V$ is the vertex set of the Michon tree associated with $(C,d)$, $[v]$ is the ultrametric disc in $C$ associated with vertex $v\in V$,
and 
\[
\sigma_1=\begin{pmatrix}0&1\\1&0\end{pmatrix}
\]
is the first Pauli matrix. Operator $D$ is the Dirac operator of the spectral triple $(\mathscr{A},\mathcal{H},D)$.
\newline

According to \cite[\S 6.1]{PB2009}, the absolute $\absolute{D}$ of the Dirac operator $D$ exists and has the form
\[
\absolute{D}\psi (v)=\diam([v])^{-1}\psi(v)
\]
and is invertible, yielding the $\zeta$-function
\[
\zeta(s)=\frac12\Trace\left(\absolute{D}^{-s}\right)
=\sum\limits_{v\in V}\diam([v])^s
\]
as a Dirichlet series. Cantor set $(C,d)$ is called \emph{$\zeta$-regular}, if $\zeta(s)$ converges on the half plane $\Re(s)>s_0$ for some $s_0\in\mathds{R}$ (called the \emph{abscissa of convergence}), and if for any $f\in\mathcal{C}(C)$ and $\tau\in \Upsilon(C)$, the limit
\[
\lim\limits_{s\downarrow s_0}(s-s_0)\Trace\left(\absolute{D}^{-s}\pi_\tau(f)\right)
\]
exists. Here, $\tau\colon V\to C\times C$ takes any vertex $v$ to a pair $(x,y)\in [v]\times [v]$ such that 
\[
\diam([v])=d(x,y),
\]
and is called a \emph{choice function}. The set of choice functions for $C$ is denoted as $\Upsilon(C)$. The expression $\pi_\tau$ is a certain represenation of $\mathcal{A}$ associated with $\tau\in\Upsilon(C)$ as follows:
\newline

Given a choice function $\tau\in\Upsilon(C)$,  representation
\[
\pi_\tau\colon\mathcal{A}\to\mathcal{B}(\mathcal{H})
\]
is constructed via
\[
\pi_\tau(f)\psi(v)
=\diag\left(f(\tau_+(v)),f(\tau_-(v))\right)\psi(v),
\]
where 
\[
\tau(v)=(\tau_+(v),\tau_-(v))
\]
for $v\in V$. This representation is a faithful $*$-representation \cite[Prop.\ 7]{PB2009}.
The measure $\mu_\tau$ on $(C,d)$ will be named \emph{Connes measure}.
\newline

Assume now that $(C,d)$ is $\zeta$-regular.
The measure on $C$ defined by \cite[\S 7.1]{PB2009} is then  given by
\begin{align}\label{ConnesMeasure}
\mu_\tau(f)=\int_C f\,d\mu_\tau
=\lim\limits_{s\downarrow s_0}\frac{\Trace\left(\absolute{D}^{-s}\pi_\tau(f)\right)}{\Trace\left(\absolute{D}^{-s}\right)}
\end{align}
for $f\in\mathcal{C}(C)$ and a fixed $\tau\in\Upsilon(C)$.
This defines a probability measure on $C$ which is independent of the choice of a choice function $\tau\in\Upsilon(C)$, cf.\ \cite[Thm.\ 3]{PB2009}.
\newline

In order to obtain an operator for a $\zeta$-regular ultrametric Cantor set $(C,d)$, it is first observed that $\mu_\tau$ induces a probability measure $\nu$ on $\Upsilon(C)$, and a Dirichlet form
\[
Q_s(f,g)
=\frac12\int_{\Upsilon(C)}
\Trace\left\{\absolute{D}^{-s}[D,\pi_\tau(f)]^*[D,\pi_\tau(g)]\right\}d\nu(\tau)
\]
on $L^2(C,\mu_\tau)$, which is closable with dense domain, cf.\ \cite[Thm.\ 4]{PB2009}. 
\newline

Supposing further that the support of $\mu_\tau$ is $C$, the theory of semigroups gives a Markovian semigroup corresponding to the Dirichlet form $Q_s$ for any $s\in\mathds{R}$ \cite[Theorem 1.4.1]{Fukushima1980}. This allows for a self-adjoint, non-positive operator $\Delta_s$ via
\[
\langle-\Delta_s f,g\rangle
=Q_s(f,g)
\]
for $f,g$ linear combinations of indicator functions of vertex discs $[v]$ in $C$, and such that $e^{t\Delta_s}$, $t\ge0$, is the corresponding Markov semigroup.
The spectrum of $\Delta_s$ is pure point, cf.\ \cite[Prop.\ 9]{PB2009}.
\newline

Generalising the observation of \cite[Prop.\ 10]{PB2009}, make the following definition.

\begin{definition}
The operator, defined as
\[
\mathcal{D}^sf(z)
=\frac12\lim\limits_{v\to z}\mu_\tau([v])^{-1}\langle 1_{[v]},-\Delta_sf\rangle
=\langle\delta_z,-\Delta_s f\rangle
\]
for $s\in\mathds{R}$ and  test functions $f\colon C\to\mathds{C}$, is called the \emph{Vladimirov-Pearson operator} on the regular Cantor set $(C,d)$.
\end{definition}

In what follows, it is important that $\mu_\tau$ be an equity measure, i.e.\ the value $\mu_\tau([v])$ is the same for all child vertices $v$ of a given vertex $v_0$.
Define for a given vertex $v$ of $T(C)$, the following quantity:
\begin{align*}
\kappa_v(s)=\sum\limits_{w\preceq v}\diam([w])^s
\end{align*}
for $s>s_0$, where
$w$ runs through the descendants of $v$, and
$s_0$ is the abscissa of convergence of $\zeta(s)$.
\newline

Sometimes,  $(C,d)$ is a $\zeta$-regular ultrametric Cantor set satisfying the condition
\begin{align}\label{factorProperty}
\kappa_v(s)=\diam([v])^{s-s_0+1}\cdot\zeta(s)
\end{align}
for all $v\in V(T(C))$, and $s>s_0$.

\begin{definition}
An ultrametric Cantor set satisfying  (\ref{factorProperty}) is said to satisfy the \emph{factorisation property}.
\end{definition}

\begin{example}[$p$-adic integers]\label{Zp}
Let
$C=\mathds{Z}_p$, and $[v]=p^\ell\mathds{Z}_p$. Then
\[
\zeta(s)=\sum\limits_{k=0}^\infty p^kp^{-ks}
\]
with abscissa of convergence 
$s_0=1$. The factorising property  (\ref{factorProperty}) is satisfied for the $p$-adic distance, since
\[
\kappa_v(s)=p^{-\ell s}\zeta(s)
\]
for $s>1$.
Hence,
\[
\mu_\tau([v])=\lim\limits_{s\downarrow 1}\frac{\kappa_v(s)}{\zeta(s)}=p^{-\ell}
\]
yields the usual Haar measure on $\mathds{Z}_p$.
\end{example}

\begin{example}[Level-regularity]\label{levelRegular}
Let $T$ be an infinite rooted tree without dangling vertices, and assume that at for the vertices at a given level $\ell$ (i.e.\ distance from root), the number of child nodes  always takes the same value $p_\ell\in\mathds{N}$. We call this \emph{level-regularity}. Assume that the ultrametric $d(x,y)$ on $\partial T$ also depends only on the level of $x\wedge y$, and takes the value
\begin{align}\label{productEll}
d(x\wedge y)=\diam([v_\ell])=p_0^{-1}\cdots p_{\ell-1}^{-1}\,,
\end{align}
where the level of $x\wedge y$ equals $\ell$, and if $\ell=1$, the product (\ref{productEll}) is the empty product equal to one. 
In this case, it holds true that
\[
\zeta(s)=\sum\limits_{k=0}^\infty p_0\cdots p_{k-1}\diam(v_k)^s=\sum\limits_{k=0}^\infty
p_0\cdots p_{k-1}\cdot p_0^{-s}\cdots p_{k-1}^{-s}
\,,
\]
where $v_k\in V(T)$ is any vertex at level $k\in\mathds{N}$.
Then
\begin{align*}
\kappa_v(s)&=\sum\limits_{k=\ell}^\infty 
p_0\cdots p_{k-\ell-1}
 \diam(v_k)^s
 \\
& =\sum\limits_{k=0}^\infty
 p_0\cdots p_{k-1}
 \cdot p_0^{-s(1+\ell)}\cdots p_{k-1}^{-s(1+\ell)}
 \\
& =p_0^{-s}\cdots p_{\ell-1}^{-s}\cdot\zeta(s)
=\diam([v_\ell])^s\cdot\zeta(s)\,,
\end{align*}
and  (\ref{factorProperty}) is satisfied, provided $s_0=1$.
\end{example}

\begin{definition}
An ultrametric  Cantor set $(C,d)$ set is said to be \emph{equitising}, if the diameters of the children of a given vertex of $T(C)$ are all the same. 
\end{definition}

\begin{assumption}\label{factorUltrametric}
It is assumed that $(C,d)$ is an equitising $\zeta$-regular ultrametric Cantor set.
\end{assumption}

\begin{remark}
Any ultrametric $d(x,y)$ on the boundary $\partial T$ of a rooted tree $T$ without dangling nodes, which depends only on the geodesic path from root to the parent vertex of   $x\wedge y\in V(T)$ for $x\neq y$ in $\partial T$, yields an equitising ultrametric Cantor set. Examples \ref{Zp} and \ref{levelRegular} provide examples of equitising ultrametric Cantor sets.
\end{remark}

\begin{thm}\label{LocalCoordinateOperator}
Let $(C,d)$ be an equitising $\zeta$-regular Cantor set. Then $\mu_\tau$ is an equity measure, and
the Vladimirov-Pearson operator on the space of test functions $\mathcal{D}(C)$ has the form
\[
\mathcal{D}^sf(z)
=\int_{C}
\frac{d(z,y)^{s-3}(f(z)-f(y))}{\mu_\tau(S_{d(z,y)}(z))}\,d\mu_\tau(y)
\]
for $s\in\mathds{R}$,
where $a\wedge b$ is the join of $a,b\in C$ in the Michon tree $T(C)$, $\children(v)$ the set of child nodes of vertex $v\in V$, and $S_\epsilon(x)$ the sphere of radius $\epsilon>0$ centred in $x\in C$.
\end{thm}

\begin{proof}
Let $\mu=\mu_\tau$. First,  observe immediately  that the equitising property implies that
$\mu$ is an equity measure.

\smallskip
From 
\cite[\S 8.3]{PB2009}
take that
\begin{align*}
\langle-\Delta_s 1_{[v]},f\rangle
&=\sum\limits_{w\succ v}
\frac{\diam([w])^{s-2}}{\sum\limits_{(u,u')\in\mathcal{G}_w}\mu([u])\mu([u'])}
\\
&\cdot 2\int_{[v]}d\mu(x)
\int_{[w]\cap[u_v]^c}
(f(x)-f(y))\,d\mu(y),
\end{align*}
where for $w\succ v$, vertex $u_v$ is the child of $w$ which is also an ancestor of vertex $v$,
and where $\mathcal{G}_w$ is the set of distinct pairs of children of vertex $w\in V$.
Since $\mu=\mu_\tau$ is an equity measure, it holds true that
\begin{align*}
\sum\limits_{(u,u')\in\mathcal{G}_w}\mu([u])\mu([u'])&=
\absolute{\children(z\wedge y)}
\left(\absolute{\children(z\wedge y)}-1\right)\cdot\mu([u_v])^2
\\
&=\mu([w])\mu\left(S_{\diam([w])}(z)\right)
\\
&=\diam([w])\mu\left(S_{\diam([w])}(z)\right)
\end{align*}
for $w\succ v$ with $z\in[v]$, and where $S_r(z)\subset C$ is the ultrametric sphere of radius $r>0$ (assumed to be a possible distance value) centred in $z\in C$.
Hence,
\begin{align*}
\langle-\Delta_s1_{[v]},g\rangle&
=\sum\limits_{w\succ v}
\frac{\diam([w])^{s-3}}{\mu\left(S_{\diam([w])}(z)\right)}\cdot
2\int_{[v]}d\mu(x)
\int_{[w]\cap[u_v]^c}(f(z)-f(x))\,d\mu(y)
\end{align*}
It now follows that
\begin{align*}
\frac{1}{2\mu([v])}\langle-\Delta_s1_{[v]},g\rangle
\end{align*}
converges for $v\to z$ to
\begin{align*}
\mathcal{D}^sf(z)&=
\sum\limits_{w\succ z}\frac{\diam([w])^{s-3}}{\mu\left(S_{\diam([w])}(z)\right)}
\cdot\int_{S_{\diam([w])}(z)}
(f(z)-f(y))\,d\mu(y)
\\
&=\int_{C\setminus\mathset{z}}\frac{d(z,y)^{s-3}}{\mu\left(S_{d(z,y)}(z)\right)}
(f(z)-f(y))\,d\mu(y),
\end{align*}
where the last equality follows from integrating over
 $y\in[w]$ such that $w=y\wedge z$ for all $w\succ z$ which covers all of $C$.
This proves the assertion, since $f\in\mathcal{D}(C)$ is constant near $z\in C$.
\end{proof}

Ultrametric wavelets are defined in \cite[Eq.\ (6)]{XK2005}, and can be written
in the setting of an ultrametric Cantor set $(C,d)$ as
\[
\psi_{v,j}(x)=\mu([v])^{-\frac12}e^{2\pi ij x_v/s_v}1_{[v]}(x)\,,
\] 
for $x\in C$,
where $v\in V(T(C))$ is a vertex,
$j=1,\dots,s_v$, and $s_v$ is the number of children of the parent vertex of $v$ in $T(C)$. The quantity $x_v\in\mathset{1,\dots,s_v}$ is defined via a numbering of the children of the parent vertex of vertex $v$ in the rooted tree $T(C)$.

\begin{Cor}\label{LocalCoordinateEigenvalue}
Under Assumption \ref{factorUltrametric},
the ultrametric wavelets are eigenfunctions of $\mathcal{D}^s$.
The eigenvalue corresponding to an ultrametric wavelet $\psi_{v,j}$ with $v\succ z\in C$ and $j=1,\dots,s_v$  equals
\[
\lambda_v = \frac{\mu_\tau([v])}{1-s_{v}^{-1}}
\cdot d(v,v^\sigma)^{s-4}
+\int_{C\setminus[v]}\frac{d(z,y)^{s-4}}{1-p_{y\wedge z}^{-1}}\,d\mu_\tau(y)
\]
where $v^\sigma$ is any sibling of $v$  distinct from $v$ in $T(C)$, and $p_w$ equals the number of children of vertex $w$ $T(C)$.
\end{Cor}

\begin{proof}
Since the kernel function of the integral operator $\mathcal{D}^s$ depends on the distance $d(y,z)$, cf.\ Theorem \ref{LocalCoordinateOperator}, the result
\cite[Theorem 10]
{XK2005} can be applied to the setting here. The statement  about the ultrametric wavelet eigenvalue there first translates to
\[
\lambda_{v}
=\mu_\tau([v])\frac{d(v,v^\sigma)^{s-3}}{\mu_\tau(S_{d(z,v^\sigma)}(z))}
+\int_{C\setminus[v]}
\frac{d(z,y)^{s-3}}{\mu_\tau\left(S_{d(z,y)}(z)\right)}\,d\mu_\tau(y)\,,
\]
using Theorem \ref{LocalCoordinateOperator}.
Now, 
calculate the value of $\mu_\tau(S_{d(x,z)}(z))$ for $x\neq z$ in $C$. By the equitising property, this value is immediately seen to be
\begin{align}\label{circleMeasure}
\mu_\tau(S_{d(x,z)}(z))=\left(1-p_{x\wedge z}^{-1}\right)d(x,z)\,,
\end{align}
from which the asserted eigenvalue follows.
Notice that $\lambda_v$ does not depend on the choices of $z\in [v]$, sibling $v^\sigma$, and also not on $j\in\mathset{1,\dots,s_v}$.
\end{proof}

Observe from (\ref{circleMeasure}) 
that the operator $\mathcal{D}^s$ can also be written as
\begin{align}\label{LocalCoordinateLaplacian}
\mathcal{D}^sf(x)=
\int_C\frac{d(x,y)^{s-4}}{1-p_{x\wedge y}^{-1}}(f(x)-f(y))\,d\mu_\tau(y)
\end{align}
for $x\in C$, $s\in\mathds{R}$,
and test functions $f\in\mathcal{D}(C)$,
by inserting into the integral representation of Theorem \ref{LocalCoordinateOperator}.

\begin{thm}\label{LocalCoordinateSpectrum}
Let $(C,d)$ be an equitising $\zeta$-regular Cantor set.
The operator $\mathcal{D}^s$ acting on $L^2(C,\mu_\tau)$ is self-adjoint and positive semi-definite.
The Hilbert space $L^2(C,\mu_\tau)$ has an orthonormal basis consisting of eigenfunctions of $\mathcal{D}^s$. The operator $\mathcal{D}^s$ is bounded if and only if $s\ge 4$. In this case it is a compact operator, and $0$ is an accumulation point of the spectrum of $\mathcal{D}^s$. In the case that $s<4$, the spectrum of $\mathcal{D}^s$ is a point spectrum.
\end{thm}

\begin{proof}
Corollary \ref{LocalCoordinateEigenvalue} shows that the ultrametric wavelets, together with  the constant function $\mu_\tau(C)^{-\frac12}$, form an orthonormal basis of $L^2(C,\mu_\tau)$ consisting of eigenfunctions, where the orthonormality relations between the wavelets are shown in \cite[Theorem 2]{XK2005}, and together with their mean zero property (which is immediate because of the equity property of $\mu_\tau$) yields the asserted orthonormal basis property. This implies the unitarily diagonalisability of $\mathcal{D}^s$ for $s\in\mathds{R}$.

\smallskip
From (\ref{LocalCoordinateLaplacian}),
it follows that the kernel function of $\mathcal{D}^s$ according to Theorem \ref{LocalCoordinateOperator} is bounded if and only if $s\ge 4$.
In this case, Corollary \ref{LocalCoordinateEigenvalue} shows that $0$ is an accumulation point in the spectrum of $\mathcal{D}^s$.

\smallskip
Self-adjointness and positive semi-definiteness
of $\mathcal{D}^s$ for $s\in\mathds{R}$ follow from the  eigendecomposition of  $L^2(C,\mu_\tau)$, and the symmetry of the kernel function.

\smallskip
Since according to Lemma \ref{LocalCoordinateEigenvalue}, the multiplicity of each eigenvalue is finite, and $L^2(C,\mu_\tau)$ is a Hilbert space, compactness follows in the case $s\ge 4$. In the other case of $s<4$, the operator $\mathcal{D}^s$ is unbounded, and its spectrum  contains no accumulation point, hence is a point spectrum. This proves the remaining assertions.
\end{proof}

\section{Vladimirov-Pearson semigroup}

Here, the semigroup associated with a Vladimirov-Pearson operator $\mathcal{D}^s$ is studied.

\begin{Lemma}
The semigroup $e^{-t\mathcal{D}^s}$ acts compactly on $L^2(C,\mu_\tau)$ for $t>0$, where $s<4$.
\end{Lemma}

\begin{proof}
If $s<4$, then $e^{-t\mathcal{D}^s}$ is clearly of trace class for $t>0$. Hence, the operators operate compactly on the Hilbert space $L^2(C,\mu_\tau)$ in this case. 
\end{proof}

Define the following Sobolev space:
\[
W^{1,2}(C)=\mathset{f\in L^2(C,\mu_\tau)\colon \norm{\mathcal{D}^sf}_{L^2(C,\mu_\tau)}<\infty}
\]
for $s\in\mathds{R}$.
The Sobolev norm on $W^{1,2}(C)$ is given by
\[
\norm{f}_{W^{1,2}(C)}=\left(\norm{f}_{L^2(C,\mu_\tau)}^2+\norm{\mathcal{D}^sf}_{L^2(C,\mu_\tau)}^2\right)^{\frac12}
\]
for $f\in W^{1,2}(C)$.

\begin{Proposition}
The Sobolev space $W^{1,2}(C)$ is a Hilbert space.
\end{Proposition}

\begin{proof}
The proofs of \cite[Proposition 4.3 and Corollary 4.4]{ellipticBVP}
carries over to the situation here in a  simplified manner.
\end{proof}

\begin{thm}\label{MarkovSemigroup}
The operator $\mathcal{D}^s$ is the generator of a strongly continuous
semigroup $e^{-t\mathcal{D}^s}$ on $W^{1,2}(C)$
for $t\ge0$, which satisfies the Markov property.
\end{thm}

\begin{proof}
Since $-t\mathcal{D}^s$ acts on the Hilbert space $L^2(C,\mu_\tau)$ for $t\ge0$, and its eigenvalues are bounded from above (they are non-positive)  by Theorem \ref{LocalCoordinateSpectrum}, it follows  that $e^{-t\mathcal{D}^s}$ is a strongly continuous semigroup acting on $L^2(C,\mu_\tau)$ for $t\ge0$.

\smallskip
The semigroup $e^{-t\mathcal{D}^s}$ with $t\ge0$ is also a contraction semigroup, because
\[
\norm{\int_0^1 e^{-\tau\mathcal{D}^s}u\,d\tau}_{L^2}\le t\norm{u}_{L^2}\,,
\]
which holds true for eigenfunctions of $\mathcal{D}^s$, and then by Pythagoras for any $u\in L^2(C,\mu_\tau)$. The contraction semigroup property then follows by the same reasoning as in the proof of \cite[Theorem 6.3]{ellipticBVP}.

\smallskip
Following further the proof of that theorem, one observes first the conditions
\begin{align*}
f\ge0\;\text{a.e.}&\quad\Rightarrow\quad e^{-t\mathcal{D}^s}f\ge0\;\text{a.e.}
\\
f\le1\;\text{a.e.}&\quad\Rightarrow\quad e^{-t\mathcal{D}^s}f\le1\;\text{a.e.}
\\
e^{-t\mathcal{D}^s}1_C&=1_C\,,
\end{align*}
where the first two hold true due to the invariance of eigenfunctions under the action of $\mathds{F}_p^\times$, and the third property holds true, because $1_C$ is an eigenfunction of $\mathcal{D}^s$ corresponding to eigenvalue zero. 

\smallskip
Then, an invariant needs to be found on $W^{1,2}(C)$. Such a measure $\pi_\lambda$ exists for each eigenspace $E_\lambda$ of $\mathcal{D}^s$, because $E_\lambda$ is finite-dimensional and $\mathcal{D}^s$ is diagonal on $E_\lambda$.
Define
\[
\pi=\sum\limits_\lambda\pi_\lambda
\]
formally, and let 
\[
f=\sum\limits_{\lambda}f_\lambda\phi_\lambda
\]
be the eigendecomposition of $f\in\mathcal{D}(C)$ with $f_\lambda\in\mathds{C}$, and observe that
\[
e^{-t\mathcal{D}^s}\pi f
=\sum_\lambda e^{-t\mathcal{D}^s}\pi_\lambda f_\lambda\phi_\lambda
=\sum_\lambda\int_C f_\lambda\phi_\lambda\,d\pi_\lambda
=\int_Cf\,d\pi\,,
\]
which shows that $\pi$ is a distribution on $\mathcal{D}(C)$. In order to see that it is also one on $W^{1,2}(C)$, the method in the proof of \cite[Theorem 6.3]{ellipticBVP} simplifies here as
\[
\absolute{\int_Cf\,d\pi}^2
=\sum\limits_\lambda 
\absolute{\langle f_\lambda\phi_\lambda,\pi_\lambda\rangle}^2
=\sum\limits_\lambda
\lambda^2\absolute{f_\lambda}^2
\le \norm{f}_{W^{1,2}(C)}^2\,,
\]
which means that 
\[
\sum\limits_\lambda e^{-t\mathcal{D}^s}\pi_\lambda\in W^{1,2}(C)'
\]
is a distribution which coincides with
\[
e^{-t\mathcal{D}^s}\pi
\]
together with the identity
\[
\int_Cf\,d\pi=
\left(\sum\limits_\lambda e^{t\mathcal{D}^s}\pi_\lambda\right) f\,.
\]
Hence, $\pi$ is the distribution on $W^{1,2}(C)$ invariant under $e^{-t\mathcal{D}^s}$ for $t\ge0$. This proves the assertions.
\end{proof}

The analogue of \cite[Corollary 6.4]{ellipticBVP} holds true:

\begin{Cor}\label{kernelRepresentation}
The semigroup $e^{-t\mathcal{D}^s}$ with $t\ge0$ acting on $W^{1,2}(C)$ has a kernel representation $p_t(x,\cdot)$ for $t\ge0$, $x\in C$, i.e.\ the map $A\mapsto p_t(x,A)$ is a Borel measure and 
\[
\int_C p_t(x,dy)\,f(y)=e^{-t\mathcal{D}^s}f(x)
\]
for $f\in W^{1,2}(C)$, and $x\in C$.
\end{Cor}

\begin{proof}
The proof of \cite[Corollary 6.4]{ellipticBVP} carries over to this case.
\end{proof}

\begin{thm}\label{LocalCoordinateHeatKernel}
Assume that $s<4$. Then the Markov semigroup $e^{-t\mathcal{D}^s}$ on $W^{1,2}(C)$ has a heat kernel function $H(t,x,y)\in L^2(C)\otimes L^2(C)$ for $t>0$, where
\[
H(t,x,y)=\sum\limits_\lambda e^{-\lambda t}\phi_\lambda(x)\overline{\phi_\lambda(y)}
\] 
with $x,y\in C$, $\lambda$ running through the eigenvalues of $\mathcal{D}^s$, and $\phi_\lambda$ an orthonormal eigenbasis in $\mathcal{D}(C)$ of $W^{1,2}(U)$.
\end{thm}

\begin{proof}
First observe that the ultrametric wavelets 
belong to $\mathcal{D}(C)$ and together with the normalised constant function form an eigenbasis of $L^2(C,\mu_\tau)$ according to Corollary \ref{LocalCoordinateEigenvalue} and Theorem \ref{LocalCoordinateSpectrum}.

\smallskip
Since $s<4$, it 
further follows that 
\[
H(t,x,x)=\Trace
\left(e^{-t\mathcal{D}^s}\right)<\infty
\]
for $x\in C$, $t>0$. Assume now that $x\neq y$, $t>0$. Since
\begin{align}\label{KozyrevBound}
\absolute{\phi_\lambda(x)
\overline{\phi_\lambda(y)}}\le \mu_\tau(C)^{-1}\,,
\end{align}
it also follow that
\[
\absolute{H(t,x,y)}\le \mu_\tau(C)^{-1}\Trace\left(e^{-t\mathcal{D}^s}\right)<\infty
\]
for $x\neq y$ in $C$. This shows that $H(t,x,y)$ is well-defined for $t>0$. In fact, 
\[
\norm{H(t,\cdot,\cdot)}_{L^2(C,\mu_\tau)^{\otimes2}}
=
\sum\limits_{\lambda}e^{-2\lambda t}<\infty
\]
shows that $H(t,\cdot,\cdot)\in L^2(C,\mu_\tau)^{\otimes 2}$ for $t>0$ in this case, as asserted.
\end{proof}

\begin{Cor}\label{GreenFunction}
Assume that $s<4$. 
Then the Green function associated with $-\mathcal{D}^s$ exists and has the form
\[
G(x,y)=\sum\limits_{\lambda>0}\lambda^{-1}\phi_\lambda(x)\overline{\phi_\lambda(y)}
\]
for $x,y\in C$, provided the sum
\[
\sum\limits_{\lambda>0}\lambda^{-1}
\]
over the positive eigenvalues of $\mathcal{D}^s$ is convergent.
\end{Cor}

\begin{proof}
The expression for $G(x,y)$ comes from
\[
G(x,y)=\int_0^\infty h(t,x,y)\,dt\,,
\]
where
\[
h(t,x,y)=\sum\limits_{\lambda>0}e^{-t\lambda}\phi_\lambda(x)\overline{\phi_\lambda(y)}\,,
\]
where $\lambda$ runs through the positive eigenvalues of $\mathcal{D}^s$ for $s<4$. By the convergence assumption for the sum of inverse eigenvalues, it now follows the Green function exists for $x,y\in C$, because also of the bound (\ref{KozyrevBound}). This proves the assertion.
\end{proof}

\begin{remark}\label{niceEigenvalue}
The ultrametric wavelet eigenvalue $\lambda_v$ can also be written as
\[
\lambda_v=\frac{d(v,v^\sigma)^{s-4}}{(1-p_{v\wedge v^\sigma}^{-1})s_v} + \int_{C\setminus[v]}
\frac{d(z,y)^{s-4}}{1-p_{y\wedge z}^{-1}}\,d\mu_\tau(y)
\]
with the notation from Corollary \ref{LocalCoordinateEigenvalue}. This follows from a simple calculation. Hence,
\[
\lambda_v\in O\left(s_v^{-1}d(v,v^\sigma)^{s-4}\right)
\]
assymptotically for $v\to z\in C$.
This means that for $s<4$,  the existence of a Green function according to Corollary \ref{GreenFunction} can be guaranteed in many cases. Such a  case occurs e.g.\ if, but not only if, the number $s_v$ of children in $T(C)$ is bounded, but also in the case of the Baire distance
\[
d(x,y)=2^{-\ell(x\wedge y)}
\]
with $\ell(x\wedge y)$ being the length of the largest common prefix of $x,y\in C$ in the rooted tree $T(C)$, and
where $s<4$ is such that the growth of $s_v$ for vertex $v$ approaching $\partial T(C)$ can be overcompensated. 
\end{remark}





\section*{Acknowledgements}
David Weisbart, \'Angel Mor\'an Ledezma, Photini Halwas and Wilson Z\'u\~niga-Galindo are thanked for fruitful discussions.
This research is partially funded by the Deutsche Forschungsgemeinschaft under 
project number 469999674.

\bibliographystyle{plain}
\bibliography{biblio}

\end{document}